\documentclass[12pt]{amsart}
\usepackage{a4wide,amsmath,enumerate,xcolor,graphicx}
\usepackage{esint} 
\allowdisplaybreaks

\let\pa\partial
\let\na\nabla
\let\eps\varepsilon

\newcommand{\R}{{\mathbb R}}
\newcommand{\diver}{\operatorname{div}}

\newcommand{\dom}{\mathcal{D}}
\newcommand{\dd}{{\mathrm{d}}}
\newcommand{\dx}{{\mathrm{d}x}}

\newtheorem{theorem}{Theorem}
\newtheorem{lemma}[theorem]{Lemma}
\newtheorem{proposition}[theorem]{Proposition}
\newtheorem{remark}[theorem]{Remark}

\begin{document}

\title[Long-time asymptotics for cross-diffusion systems]{Large-time asymptotics for general cross-diffusion systems and novel convex Sobolev inequalities}

\author[N. Geltner]{Noah Geltner}
\address{Institute of Analysis and Scientific Computing, TU Wien, Wiedner Hauptstra\ss e 8--10, 1040 Wien, Austria}
\email{noah.geltner@tuwien.ac.at} 

\author[A. J\"ungel]{Ansgar J\"ungel}
\address{Institute of Analysis and Scientific Computing, TU Wien, Wiedner Hauptstra\ss e 8--10, 1040 Wien, Austria}
\email{juengel@tuwien.ac.at} 

\date{\today}

\thanks{The authors acknowledge partial support from the Austrian Science Fund (FWF), grant 10.55776/PAT2687825, and from the Austrian Federal Ministry for Women, Science and Research and implemented by \"OAD, project MULT09/2025. This work has received funding from the European Research Council (ERC) under the European Union's Horizon 2020 research and innovation programme, ERC Advanced Grant NEUROMORPH, no.~101018153. For open-access purposes, the authors have applied a CC BY public copyright license to any author-accepted manuscript version arising from this submission.} 

\begin{abstract}
The exponential decay of weak solutions towards the constant steady state of general cross-diffusion systems in bounded domains with no-flux boundary conditions is investigated. The proof of quantitative decay rates is based on the relative entropy method and involves two parameters: the exponent determining the entropy density and the exponent in the entropy production integral. The key ingredient are novel convex Sobolev inequalities, which extend existing results in the literature to a broad range of values of the two parameters. These inequalities are established using convexity arguments and the Gagliardo--Nirenberg inequality.
\end{abstract}

 

\maketitle


\section{Introduction}

The long-time behavior of solutions to cross-diffusion systems is a fundamental question in the analysis of multicomponent diffusion processes. In contrast to classical diffusion equations, the coupling between the different components may lead to a diffusion matrix that is neither symmetric nor positive definite, making the derivation of quantitative convergence rates particularly challenging. A powerful framework for studying such systems is provided by the entropy structure, which allows one to exploit the underlying dissipative mechanism \cite{Jue16}. In this paper, we develop a general approach to proving exponential convergence towards equilibrium based on the relative entropy method. The key ingredient are novel convex Sobolev inequalities, extending those from \cite{AbLe24} and providing a quantitative relation between the relative entropy and its associated dissipation. Combined with the entropy inequality, this leads to a differential inequality for the relative entropy and, by Gronwall's inequality, to the exponential decay of the solutions towards equilibrium. Our approach applies to a broad class of cross-diffusion systems and highlights the interplay between the entropy structure and convex functional inequalities in determining the rate of relaxation.

\subsection{Model setting}

We consider cross-diffusion systems of the type
\begin{align}
  & \pa_t u_i = \diver\bigg(\sum_{j=1}^n A_{ij}(u)\na u_j\bigg)
  \quad\mbox{in }\Omega,\ t>0,\ i=1,\ldots,n, \label{1.eq} \\
  & u_i(0)=u_i^0\quad\mbox{in }\Omega, \quad
  \sum_{j=1}^n A_{ij}(u)\na u_j\cdot\nu = 0\quad\mbox{on }\pa\Omega,
  \ t>0, \label{1.bic}
\end{align}
where $\Omega\subset\R^d$ ($d\ge 1$) is a bounded domain, $\nu$ is the exterior unit normal vector to $\pa\Omega$, $u=(u_1,\ldots,u_n)$ is the solution vector, and $A_{ij}(u)$ are diffusion coefficients. The variables $u_i$ describe particle densities or volume fractions of multicomponent mixtures. In many applications, the diffusion matrix $A(u)=(A_{ij}(u))$ is neither symmetric nor positive definite, but equations \eqref{1.eq} possess an entropy structure \cite{Jue16}. This means that there exists a convex function $h:\dom\to\R$, called an entropy density, where $\dom\subset(0,\infty)^n$ is an open set, and numbers $c_A>0$ and $s>0$ such that for all $u\in\dom$ and $z\in\R^n$,
\begin{align}\label{1.posdef}
  z^Th''(u)A(u)z \ge c_A\sum_{i=1}^n u_i^{2s-2}z_i^2,
\end{align}
where $h''(u)$ denotes the Hessian of $h$. The existence of global weak solutions to \eqref{1.eq}--\eqref{1.bic} was proved in \cite{Jue15} under suitable assumptions on the data. Examples include the Maxwell--Stefan equations \cite{JuSt13} and thin-film solar-cell systems \cite{BaEh18} with $s=1/2$, Shigesada--Kawasaki--Teramoto (SKT) models \cite{ChJu04} with $s=1$, and general population models \cite{CDJ18} with $s>0$. In this paper, we restrict our attention to entropy densities of the form 
\begin{align}\label{1.h}
  h_1(u) = \sum_{i=1}^n u_i(\log u_i-1), \quad
  h_q(u) = \sum_{i=1}^n\frac{u_i^q-u_i}{q-1}\quad\mbox{for }q>1,
\end{align}
which arise in many applications, in particular in those mentioned above. The function $h_1$ is called the Boltzmann--Shannon entropy density, while $h_q$ refers to the Tsallis entropy density \cite{Tsa88}. Observe that $h_q$ converges pointwise to $h_1$ as $q\to 1$. 

The unique steady state associated to \eqref{1.eq} is constant and given by $\bar u=(\bar{u}_1,\ldots,\bar{u}_n)$, where
\begin{align*}
  \bar u_i^0 := \fint_\Omega u_i^0 \dx  
  = \frac{1}{|\Omega|}\int_\Omega u_i^0 \dx , \quad i=1,\ldots,n.
\end{align*}
A natural question is to ask how fast the solution $u(t)$ to \eqref{1.eq}--\eqref{1.bic} converges to this steady state. It turns out that in many cross-diffusion systems, this convergence is exponential. The proof is based on the so-called relative entropy method, using the relative entropy density
\begin{align}\label{1.hrel}
  h(u|\bar u) = h(u) - h(\bar u) - h'(\bar u)\cdot(u-\bar u).
\end{align}
A formal computation, using \eqref{1.posdef}, shows that
\begin{align*}
  \frac{\dd}{\dd t}\int_\Omega h(u|\bar u)\dx 
  = -\int_\Omega\na u^T:h''(u)A(u)\na u\dx  
  \le -\frac{c_A}{s^2}\sum_{i=1}^n\int_\Omega|\na u_i^s|^2 \dx .
\end{align*}
If $q=1$ and $s=1/2$, the logarithmic Sobolev inequality 
\begin{align}\label{1.lsi}
  \int_\Omega h_1(u|\bar u)\dx  
  = \int_\Omega u_i\log\frac{u_i}{\bar{u}_i}\dx 
  \le C_{LS}\int_\Omega|\na\sqrt{u_i}|^2 \dx 
\end{align}
then leads to 
\begin{align*}
  \frac{\dd}{\dd t}\int_\Omega h_1(u|\bar u)\dx 
  + \frac{c_A}{s^2 C_{LS}}\int_\Omega h_1(u|\bar u)\dx  \le 0.
\end{align*}
Gronwall's inequality implies the exponential decay
\begin{align*}
  \int_\Omega h_1(u(t)|\bar u)\dx  \le e^{-\lambda t}
  \int_\Omega h_1(u^0|\bar u)\dx , \quad t>0,
\end{align*}
where the decay rate $\lambda=c_A/(s^2C_{LS})$ is semi-explicit. By the Csisz\'ar--Kullback inequality \cite{Csi67,Kul67} (see Proposition \ref{prop.cki}), we deduce the exponential decay of $u(t)$ towards $\bar u$ in the $L^1(\Omega)$ norm with rate $\lambda/2$. For general values of $q$ and $s$, convex Sobolev inequalities of the type
\begin{align}\label{1.ineq} 
  \int_\Omega h_q(u|\bar u)\dx  
  \le C_S\sum_{i=1}^n\int_\Omega|\na u_i^s|^2 \dx 
\end{align}
are required. This inequality has been proved up to now for special values of $(q,s)$ only. For instance, $(q,s)=(2,1)$ corresponds to the Poincar\'e--Wirtinger inequality, and the case $(q,s)=(1,1)$ was analyzed in \cite[Theorem 1]{AbLe24}. The choice $1<q<2$, $s=q/2$ yields the Beckner inequality \cite{Bec89}, extended to $s>q/2$ in \cite[Lemma 3.6]{CJP16}. General convex functions $\psi(z)$ (instead of $z^q$) such that $1/\psi''$ is concave have been investigated in \cite[Remark 3.8]{AMTU01}. The aim of this paper is to generalize inequality \eqref{1.ineq} to a broad range of admissible pairs $(q,s)$ and to deduce exponential decay rates for the solutions to \eqref{1.eq}--\eqref{1.bic}. 


\subsection{State of the art}

In principle, the exponential decay of solutions to the steady state can be established by estimating the spectral gap to the associated differential operator. However, this technique is essentially limited to linear problems. The relative entropy provides a nonlinear measure of the distance between two solutions that is particularly well suited to nonlinear problems. This concept was first used in \cite{Daf79} as a mathematical tool to prove uniqueness and continuous dependence of smooth thermodynamic processes. The relative entropy method was extended to establish exponential decay to equilibrium for diffusion equations in \cite{AMTU01,CaTo00}, and subsequently to diffusion systems with diagonal diffusion matrices in \cite{DeFe06}. 

An early use of the relative entropy method in cross-diffusion systems can be found in \cite{ChJu06}. The logarithmic Sobolev inequality \eqref{1.lsi} allows for the determination of exponential equilibration rates in nondegenerate cross-diffusion systems, like Maxwell--Stefan systems \cite{DJT20}, thin-film solar-cell models \cite{HoBu22}, and volume-filling population systems \cite{CJLL24}. The method was generalized to Maxwell--Stefan equations with reversible reactions \cite{DJT20} and to SKT models with Lotka--Volterra terms \cite{JuZa16,Shi02}. 

A related technique is to exploit the gradient-flow structure with respect to the Wasserstein metric. For coupled multicomponent systems, this technique requires a special algebraic structure; see, e.g., \cite{LiMi13,ZiMa15}. More generally, exponential convergence results have been obtained for systems with small cross-diffusion terms \cite{BMZ23,MaPa26}. 

Here, we obtain exponential equilibration rates owing to the positive definiteness property \eqref{1.posdef}. Due to the singular or degenerate structure induced by the exponent $s$, the logarithmic Sobolev inequality \eqref{1.lsi} is not sufficient to close the argument. Instead, we require novel convex Sobolev inequalities, which are established in this paper.


\subsection{Results}

First, we introduce the relative entropy densities associated to \eqref{1.h}:
\begin{align*}
  h_1(u|\bar u) &= h_1(u) - h_1(\bar u) - h_1'(\bar u)\cdot(u-\bar u)
  = \sum_{i=1}^n u_i\log\frac{u_i}{\bar{u}_i}, \\
  h_q(u|\bar u) &=  h_q(u) - h_q(\bar u) - h_q'(\bar u)\cdot(u-\bar u)
  = \sum_{i=1}^n \frac{u_i^q-\bar{u}_i^q}{q-1}\quad\mbox{for }q>1,
\end{align*}
and we set 
\begin{align}\label{1.eta}
  \eta_1(g|\bar g) = g\log\frac{g}{\bar{g}}, \quad
  \eta_q(g|\bar g) = \frac{g^q-\bar{g}^q}{q-1}\quad\mbox{for }q>1.
\end{align}
We impose the following assumptions:
\begin{itemize}
\item[(H1)] Domains: $\Omega\subset\R^d$ ($d\ge 1$) is a bounded domain with Lipschitz boundary and $\dom \subset(0,\infty)^n$ is a domain.
\item[(H2)] Data: $u^0=(u_1^0,\ldots,u_n^0)\in L^q(\Omega;\R^n)$ for $q\ge 1$ is such that $u^0(x)\in\overline{\dom }$ for a.e.\ $x\in\Omega$.
\item[(H3)] Positive definiteness: There exist $s>0$ and $c_A>0$ such that 
\begin{align*}
  z^Th_q''(u)A(u)z \ge c_A\sum_{i=1}^n u_i^{2s-2}z_i^2 
  \quad\mbox{for all }z\in\R^n,\ u\in \dom.
\end{align*}
\end{itemize}

Assuming additionally that the set $\dom$ is bounded, the diffusion matrix is continuous on $\overline{\dom}$ and satisfies $|A_{ij}(u)u_j^{1-s}|\le C$ for some $C>0$ if $s>1$, the existence of a global weak solution $u$ to \eqref{1.eq}--\eqref{1.bic}, satisfying $u(x,t)\in\overline{\dom}$ for a.e.\ $(x,t)\in\Omega\times(0,T)$ and 
\begin{align}\label{1.regul}
  u\in L^2_{\rm loc}(0,\infty;H^1(\Omega;\R^n)), \quad
  \pa_t u\in L^2_{\rm loc}(0,\infty;H^1(\Omega;\R^n)')
\end{align}
is proved in \cite[Theorem 2]{Jue15}. Moreover, the entropy inequality
\begin{align}\label{1.ei}
  \frac{\dd}{\dd t}\int_\Omega h(u)\dx  
  + \int_\Omega\na u^T:h''(u)A(u)\na u \dx  \le 0, \quad t>0,
\end{align} 
is fulfilled \cite[Appendix  A]{HeJu26}. In particular, since  $\dom$ is bounded, the constructed weak solution is bounded too. If $\dom$ is not bounded, the result of \cite{Jue15} is not directly applicable, but its proof technique often still applies. Therefore, we assume that a weak solution satisfying \eqref{1.ei} exists. Our main result is as follows.

\begin{theorem}[Exponential decay]\label{thm.time}
Let Hypotheses (H1)--(H3) hold and let $u$ be a nonnegative weak solutions $u$ to \eqref{1.eq}--\eqref{1.bic} satisfying \eqref{1.regul} and \eqref{1.ei}. Let 
\begin{align*}
  q=1,\ s>\frac{d-2}{2d} \quad\mbox{or}\quad
  q>1,\ s\ge \frac{d-2}{2d}q
\end{align*}
and set $r=\min\{2,q\}$. Then there exists a constant $\lambda>0$, depending on the $L^q(\Omega)$ norm of $u^0$ as well as on $\Omega$, $d$, $q$, $s$, such that
\begin{align*}
  \|u-\bar{u}\|_{L^r(\Omega)}\le e^{-\lambda t}
  \bigg(\int_\Omega h_q(u^0|\bar u)\dx \bigg)^{1/2} \quad\mbox{for }t>0.
\end{align*}
\end{theorem}

The theorem is formulated for $d\ge 3$, but it also holds for $d=1,2$ with obvious modifications. The generalization to cross-diffusion equations with reaction terms $r_i(u)$ is delicate, since constant steady states $\bar u$ must also satisfy $r_i(\bar u)=0$. Thus, we need to establish a convex Sobolev inequality for elements on the manifold defined by $r(\bar u)=0$. Such an inequality was proved for Maxwell--Stefan systems with reactions in \cite{DJT20}. A general result, however, is not known. We do not consider the case $0<q<1$, since $\|\cdot\|_{L^q(\Omega)}$ is only a quasi-norm. As described above, the proof of Theorem \ref{thm.time} is based on the convex Sobolev inequality \eqref{1.ineq} and a generalized Csisz\'ar--Kullback inequality; see Proposition \ref{prop.cki}. 

\begin{theorem}[General convex Sobolev inequality]\label{thm.csi}
Let $q= 1$, $s>q(d-2)/(2d)$ or $ q >1$, $s \geq q(d-2)/(2d)$. Then there exists $C(q,s)>0$ only depending on $\Omega$, $d$, $q$, $s$ such that for all nonnegative functions $g$ with $g\in L^q(\Omega)$ and $\na g^s\in L^2(\Omega)$,
\begin{align*}
  \int_\Omega \eta_q(g|\bar g)\dx  \le C(q,s)\|g\|_{L^q(\Omega)}^{q-2s}
  \int_\Omega|\na g^{s}|^2 \dx ,
\end{align*}
where $\eta_q$ is the entropy density defined in \eqref{1.eta}. 
\end{theorem}

The case $q=1$ and $s=1/(p-2)$ for integers $p\ge 3$ was established in \cite[Theorem 1]{AbLe24}. The proof presented there is remarkably elementary and only relies on the Jensen, H\"older, and Poincar\'e--Wirtinger inequalities, together with a linear control of logarithmic terms. In this work, we extend this result to $q\ge 1$ and a broader range of values of $s$. Our proof employs similar ingredients to those used in \cite{AbLe24}, supplemented by convexity arguments for the function $G(r)=r\log(\int_\Omega g^{q/r}\dx )$ for $r\ge 1$ and the Gagliardo--Nirenberg inequality. The proof of Theorem \ref{thm.csi} proceeds through a sequence of interdependent lemmas. The logical flow of the proof is illustrated in Figure \ref{fig}. Observe that the lower bound $s\ge q(d-2)/(2d)$ is optimal; see Remark \ref{rem.opt}.

\begin{figure}[ht]
\includegraphics[width=125mm]{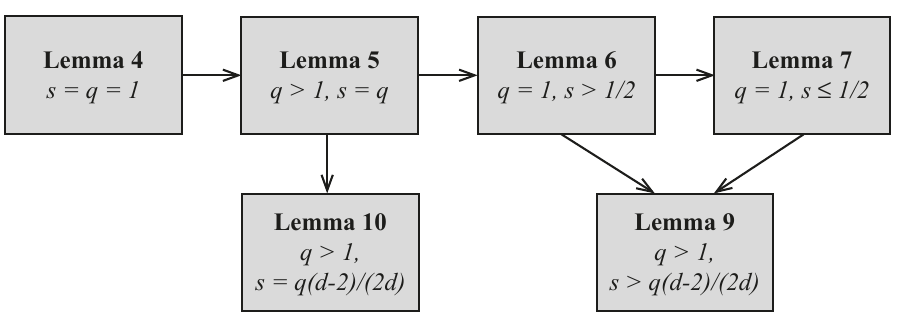}
\caption{Logical flow of the proof of Theorem \ref{thm.csi}.}
\label{fig}
\end{figure}

The convex Sobolev inequality in Theorem \ref{thm.csi} shows that the H\"older deficit 
\begin{align*}
  D(u) := \|u\|_{L^p(\Omega)}^p
  - |\Omega|^{1-q}\|u\|_{L^{p/q}(\Omega)}^q\ge 0
\end{align*}
is controlled by the Dirichlet energy of $u$. Indeed, setting $u=g^s$ and $p=q/s$, the convex Sobolev inequality is equivalent to
\begin{align*}
  D(u) \le C(q,s)\|u\|_{L^p(\Omega)}^{p-2}\|\na u\|_{L^2(\Omega)}^2. 
\end{align*}
The deficit vanishes identically on constants, and, being differentiable there, vanishes to second order. Thus, the right-hand side must be quadratic in $\na u$.

The constant $C(q,s)$ constructed in the proof is not sharp. For $1\le q\le 2$, the Bakry--Emery method may give better constants, but it does not work in the regime $q>2$, where our interpolation-theoretic arguments still work.

The paper is organized as follows. The proof of Theorem \ref{thm.csi} is given in Section \ref{sec.csi}, and the generalized Csisz\'ar--Kullback inequality is proved in Section \ref{sec.cki}. Based on these preparations, Section \ref{sec.time} is devoted to the proof of Theorem \ref{thm.time}. Finally, we apply our main theorem to two concrete cross-diffusion systems in Section \ref{sec.exam}. 


\section{Proof of Theorem \ref{thm.csi}}\label{sec.csi}

We begin with the following lemma that gives a linear control of logarithmic terms.

\begin{lemma}\label{lem.ineq}
Let $p>1$ and $a_1>0$, $a_2\ge 0$, $a_3\ge 0$. Then, for all $z\ge 0$,
\begin{align*}
  \log(1 + a_1z)\le a_1z, \quad 
  \log(1 + a_1z + a_2z^2)\le K_1z, \quad
  \log(1 + a_1z + a_2z^2 + a_3z^p)\le K_2z,
\end{align*}
where $K_1=\max\{a_1,2a_2/a_1\}$ and $K_2=\max\{a_1+a_3, 2a_2/a_1, p-1\}$.
\end{lemma}

\begin{proof}
Let $F(z) = \log(1+a_1z + a_2z^2 + a_3z^p)-Kz$ for $z\ge 0$. Then $F(0)=0$ and $F'(z)\le 0$ if and only if
\begin{align*}
  a_1+2a_2z+pa_3z^{p-1} \le K(1+a_1z+a_2z^2+a_3z^p)
\end{align*}
or if and only if
\begin{align*}
  (K-a_1) + (Ka_1-2a_2)z + Ka_2z^2 + Ka_3z^p - pa_3z^{p-1} \ge 0.
\end{align*}
If $a_2=a_3=0$, it is sufficient to choose $K=a_1$. If $a_3=0$, we take $K=K_1=\max\{a_1,2a_2/a_1\}$. If $a_3>0$, Young's inequality gives for the last two terms:
\begin{align*}
  Ka_3z^p - pa_3z^{p-1} 
  \ge Ka_3z^p - pa_3\bigg(\frac{p-1}{p}z^p+\frac{1}{p}\bigg),
\end{align*}
leading to 
\begin{align*}
  (K-a_1) &+ (Ka_1-2a_2)z + Ka_2z^2 + Ka_3z^p - pa_3z^{p-1} \\
  &\ge (K-a_1-a_3) + (Ka_1-2a_2)z + Ka_2z^2 + a_3(K-(p-1))z^p \ge 0,
\end{align*}
if $K=K_2=\max\{a_1+a_3, 2a_2/a_1, p-1\}$. In all cases, we have $F'(z)\le 0$ and hence $F(z)\le 0$ for $z\ge 0$.
\end{proof}

We formulate a variant of the logarithmic Sobolev inequality. For this, we recall the Poincar\'e--Wirtinger inequality for functions $g\in H^1(\Omega)$:
\begin{align}\label{2.pwi}
  \|g-\bar g\|_{L^2(\Omega)} \le C_P\|\na g\|_{L^2(\Omega)},
\end{align}
where $C_P>0$ only depends on the dimension $d$ and on $\Omega$. 

\begin{lemma}\label{lem.lsi}
It holds for all nonnegative functions $g\in L^1(\Omega)$ satisfying $\na g\in L^2(\Omega)$ that
\begin{align*}
  \int_\Omega g\log\frac{g}{\bar g}\dx  \le C_P^2\bar{g}^{-1}
  \int_\Omega|\na g|^2 \dx ,
\end{align*}
where $C_P>0$ is the constant of the Poincar\'e--Wirtinger inequality \eqref{2.pwi}.
\end{lemma}

\begin{proof}
The result follows from \cite[Theorem 1]{AbLe24} choosing $p=3$. For completeness, we present the short proof. First, we observe that the Poincar\'e--Wirtinger inequality and the regularity for $g$ imply that $g\in L^2(\Omega)$ and hence $g\in H^1(\Omega)$. We apply Jensen's inequality to the concave function $z\mapsto\log z$ with respect to the probability measure $g(x)\dx /(|\Omega|\bar g)$:
\begin{align}\label{2.aux}
  \int_\Omega g\log\frac{g}{\bar g}\dx 
  = |\Omega|\bar g\int_\Omega \log\frac{g}{\bar g}
  \frac{g\dx }{|\Omega|\bar g}
  \le |\Omega|\bar g\log\bigg(\int_\Omega 
  \frac{g}{\bar g}\frac{g\dx }{|\Omega|\bar g}\bigg)
  = |\Omega|\bar g\log\|f\|_{L^2(\Omega)}^2,
\end{align}
where $f = g/(\sqrt{|\Omega|}\bar g)$. We estimate the $L^2(\Omega)$ norm of $f$, using the identities $\int_\Omega(f-\bar f)\dx =0$ and $\int_\Omega \bar{f}^2 \dx = 1$ (by direct computation) as well as the Poincar\'e--Wirtinger inequality:
\begin{align*}
  \|f\|_{L^2(\Omega)}^2 = \int_\Omega(f-\bar f)^2\dx  
  + \int_\Omega\bar{f}^2 \dx  
  \le C_P^2\|\na f\|_{L^2(\Omega)}^2 + 1.
\end{align*} 
Lemma \ref{lem.ineq} with $z=\|\na f\|_{L^2(\Omega)}^2$ and $a_1=C_P^2$ shows that
\begin{align*}
  \log \|f\|_{L^2(\Omega)}^2 
  \le \log\big(C_P^2\|\na f\|_{L^2(\Omega)}^2 + 1\big)
  \le C_P^2\|\na f\|_{L^2(\Omega)}^2.
\end{align*}
We insert this estimate into \eqref{2.aux} to find that 
\begin{align*}
  \int_\Omega g\log\frac{g}{\bar g}\dx 
  \le C_P^2|\Omega|\bar g\|\na f\|_{L^2(\Omega)}^2
  = C_P^2\bar{g}^{-1}\|\na g\|_{L^2(\Omega)}^2,
\end{align*}
which finishes the proof.
\end{proof}

The previous lemma can be used to prove a similar result for $q>1$.

\begin{lemma}\label{lem.q}
Let $q>1$. It holds for all nonnegative functions $g\in L^q(\Omega)$ satisfying $\na g^q\in L^2(\Omega)$ that
\begin{align*}
  \int_\Omega(g^q-\bar{g}^q)\dx  \le C_P^2
  \bigg(\fint_\Omega g^q \dx \bigg)^{-1}\int_\Omega|\na g^q|^2 \dx ,
\end{align*}
where $C_P>0$ is the constant of the Poincar\'e--Wirtinger inequality \eqref{2.pwi}.
\end{lemma}

\begin{proof}
We proceed similarly as in \cite[Sec.~3.2]{CJP16} and define the function
\begin{align*}
  G(r) = r\log\bigg(\fint_\Omega g^{q/r}\dx \bigg)\quad\mbox{for }
  r\ge 1.
\end{align*}
The derivatives equal
\begin{align*}
  G'(r) &= \bigg(\fint_\Omega g^{q/r}\dx \bigg)^{-1}
  \bigg(\fint_\Omega g^{q/r}\dx \log\fint_\Omega g^{q/r}\dx  
  - \frac{q}{r}\fint_\Omega g^{q/r}\log g\dx \bigg), \\
  G''(r) &= \frac{q^2}{r^3}\bigg(\fint_\Omega g^{q/r}\dx \bigg)^{-2}
  \bigg(\fint_\Omega g^{q/r}\dx \fint_\Omega g^{q/r}(\log g)^2\dx 
  - \bigg(\fint_\Omega g^{q/r}\log g \dx \bigg)^2\bigg).
\end{align*}
It follows from the Cauchy--Schwarz inequality that $G''(r)\ge 0$. Thus, $G$ is convex. Next, we define
\begin{align*}
  H(r) = -\frac{e^{G(r)}-e^{G(1)}}{r-1}\quad\mbox{for }r\ge 1.
\end{align*}
Since $r\mapsto \exp G(r)$ is convex, $H$ is nonincreasing. Hence, by Lemma \ref{lem.lsi}, applied to $g^q$,
\begin{align}\label{2.H}
  H(r) &\le \lim_{r\to 1}H(r) = -G'(1)\exp G(1)
  = \fint_\Omega g^q\log\frac{g^q}{\overline{g^q}}\dx  
  \le C_P^2(\overline{g^q})^{-1}\fint_\Omega|\na g^q|^2 \dx ,
\end{align}
which, for $r=q$, is equivalent to 
\begin{align*}
  \frac{e^{G(1)}-e^{G(q)}}{q-1}
  = \frac{1}{q-1}\bigg(\fint_\Omega g^q\dx  
  - \bigg(\fint_\Omega g\dx \bigg)^q\bigg)
  \le C_P^2(\overline{g^q})^{-1}\fint_\Omega|\na g^q|^2 \dx ,
\end{align*}
which completes the proof.
\end{proof}

Lemmas \ref{lem.lsi} and \ref{lem.q} lead to the following result involving an additional parameter $s$. 

\begin{lemma}\label{lem.qs1}
Let $s>1/2$. It holds for all nonnegative functions $g\in L^1(\Omega)$ satisfying $\na g^s\in L^2(\Omega)$ that
\begin{align*}
  \int_\Omega g\log\frac{g}{\bar g}\dx  
  \le C_1(s)\bar{g}^{1-2s}\int_\Omega|\na g^s|^2 \dx,
\end{align*}
where $C_1(s)=3C_P^2/(2s-1)$. 
\end{lemma}

\begin{proof}
We choose a nonnegative function $g^s\in L^1(\Omega)$ such that $\na g^s\in L^2(\Omega)$. If $s\le 1$ then $g\in L^1(\Omega)$ implies that $g^s\in L^1(\Omega)$. Otherwise, the regularity $g^s\in L^1(\Omega)$ is an additional assumption, and we need to use an approximation argument to extend the result to functions $g\in L^1(\Omega)$. 

We proceed similarly as in the proof of Lemma \ref{lem.lsi}. By Jensen's inequality with probability measure $g(x)\dx /(|\Omega|\bar g)$,
\begin{align}\label{2.glogg}
  \int_\Omega g\log\frac{g}{\bar g}\dx 
  &= \frac{|\Omega|\bar g}{2s-1}\int_\Omega\log
  \bigg(\frac{g}{\bar g}\bigg)^{2s-1}\frac{g\dx }{|\Omega|\bar g}
  \le \frac{|\Omega|\bar g}{2s-1}\log\bigg(\frac{1}{|\Omega|}
  \int_\Omega\bigg(\frac{g}{\bar g}\bigg)^{2s}\dx \bigg) \\
  &= \frac{|\Omega|\bar g}{2s-1}\log\|f\|_{L^2(\Omega)}^2, \nonumber 
\end{align}
where $f = g^s/(\sqrt{|\Omega|}\bar{g}^s)$. The identity $\int_\Omega(f-\bar f)\dx =0$ and the Poincar\'e--Wirtinger inequality lead to
\begin{align}\label{2.aux2}
  \|f\|_{L^2(\Omega)}^2 = \int_\Omega(f-\bar f)^2 \dx 
  + \int_\Omega\bar{f}^2 \dx  
  \le C_P^2\|\na f\|_{L^2(\Omega)}^2 + \int_\Omega\bar{f}^2\dx.
\end{align}
If $s\le 1$, we deduce from Jensen's inequality that $\int_\Omega\bar{f}^2\dx = (\overline{g^s}/\bar{g}^s)^2\le 1$. Consequently, it follows from \eqref{2.aux2} that $\|f\|_{L^2(\Omega)}^2 \le C_P^2\|\na f\|_{L^2(\Omega)}^2 + 1$ and, by Lemma \ref{lem.ineq},
\begin{align*}
  \int_\Omega g\log\frac{g}{\bar g}\dx 
  &\le \frac{|\Omega|\bar g}{2s-1}
  \log\big(1 + C_P^2\|\na f\|_{L^2(\Omega)}^2\big) \\
  &\le \frac{C_P^2|\Omega|\bar g}{2s-1}\|\na f\|_{L^2(\Omega)}^2 
  = \frac{C_P^2}{2s-1}\bar{g}^{1-2s}
  \|\na g^s\|_{L^2(\Omega)}^2.
\end{align*}

Next, let $s>1$. Compared to the proof of Lemma \ref{lem.lsi}, the last term in \eqref{2.aux2} is generally not bounded from above by one and therefore needs to be estimated separately. We apply Lemma \ref{lem.q} with $q$ replaced by $s$ and use $\overline{g^s}\ge\bar{g}^s$ (by Jensen's inequality):
\begin{align*}
  \int_\Omega\bar{f}^2\dx &= \frac{1}{|\Omega|^2\bar{g}^{2s}}
  \bigg(\int_\Omega g^s \dx \bigg)^2
  = \frac{1}{|\Omega|^2\bar{g}^{2s}}\bigg(
  \int_\Omega(g^s-\bar{g}^s)\dx  + |\Omega|\bar{g}^s\bigg)^2 \\
  &\le \frac{1}{|\Omega|^2\bar{g}^{2s}}\bigg(
  C_P^2(\overline{g^s})^{-1}\int_\Omega|\na g^s|^2 \dx 
  + |\Omega|\bar{g}^s\bigg)^2 \\
  &= \big(C_P^2|\Omega|^{-1}\bar{g}^{-s}(\overline{g^s})^{-1}
  \|\na g^s\|_{L^2(\Omega)}^2 + 1\big)^2 
  \le \big(C_P^2|\Omega|^{-1}\bar{g}^{-2s}
  \|\na g^s\|_{L^2(\Omega)}^2 + 1\big)^2 \\
  &= 1 + 2C_P^2|\Omega|^{-1}\bar{g}^{-2s}\|\na g^s\|_{L^2(\Omega)}^2
  + \big(C_P^2|\Omega|^{-1}\bar{g}^{-2s}\|\na g^s\|_{L^2(\Omega)}^2
  \big)^2,
\end{align*}
We infer from \eqref{2.glogg} and Lemma \ref{lem.ineq} with  $z = C_P^2|\Omega|^{-1}\bar{g}^{-2s}\|\na g^s\|_{L^2(\Omega)}^2$, $a_1=3$, and $a_2=1$ that
\begin{align*}
  \int_\Omega g\log\frac{g}{\bar g}\dx 
  \le \frac{|\Omega|\bar g}{2s-1}\log(1 + a_1z + a_2z^2)
  \le 3z = \frac{3C_P}{2s-1}\bar{g}^{1-2s}
  \|\na g^s\|_{L^2(\Omega)}^2.
\end{align*}
This shows the lemma with $C_1(s) = C_P^2/(2s-1)$ if $s\le 1$ and $C_1(s)=3C_P^2/(2s-1)$ if $s>1$.
\end{proof}

A similar result like in Lemma \ref{lem.qs1} holds for $s<1/2$. For this, we recall the Gagliardo--Nirenberg inequality
\begin{align}\label{2.gni}
  \|f\|_{L^p(\Omega)} \le C_{GN}\big(\|\na f\|_{L^2(\Omega)}^\theta
  \|f\|_{L^2(\Omega)}^{1-\theta} + \|f\|_{L^2(\Omega)}\big)
\end{align}
for functions $f\in H^1(\Omega)$ and for some $C_{GN}>0$ depending on $\Omega$ and $p$, where $2<p\le 2d/(d-2)$, $\theta=d(p-2)/(2p)\in(0,1]$ (and $p<\infty$ if $d=2$, $p\le\infty$ if $d=1$). 

\begin{lemma}\label{lem.qs2}
Let $(d-2)/(2d)<s\le 1/2$. Then there exists $C_2(s)>0$, only depending on $|\Omega|$, $d$, $s$, $C_P$, and $C_{GN}$ such that for all nonnegative functions $g\in L^1(\Omega)$ satisfying $\na g^s\in L^2(\Omega)$,
\begin{align*}
  \int_\Omega g\log\frac{g}{\bar g}\dx  
  \le C_2(s)\bar{g}^{1-2s}\int_\Omega|\na g^s|^2 \dx,
\end{align*}
where $C_P>0$ is the constant of the Poincar\'e--Wirtinger inequality \eqref{2.pwi} and $C_{GN}>0$ is the constant of the Gagliardo--Nirenberg inequality \eqref{2.gni}.
\end{lemma}

\begin{proof}
The proof is similar to that of Lemma \ref{lem.qs1}, but we need to work in an $L^p(\Omega)$ space for some $p>2$ rather than in $L^2(\Omega)$. To this end, let $\eps>0$ and set $p=1/s+\eps>2$. By Jensen's inequality with probability measure $g(x)\dx /(|\Omega|\bar g)$,
\begin{align*}
  \int_\Omega g\log\frac{g}{\bar g}\dx 
  &= \frac{|\Omega|\bar g}{ps-1}\int_\Omega\log
  \bigg(\frac{g}{\bar g}\bigg)^{ps-1}\frac{g\dx }{|\Omega|\bar g}
  \le \frac{|\Omega|\bar g}{\eps s}\log\bigg(\frac{1}{|\Omega|}
  \int_\Omega\bigg(\frac{g}{\bar g}\bigg)^{ps}\dx \bigg) \\
  &= \frac{|\Omega|\bar g}{\eps s}\log\|f\|_{L^p(\Omega)}^p,
\end{align*}
where $f=g^s/(|\Omega|^{1/p}\bar{g}^s)$. The Gagliardo--Nirenberg inequality as well as $g^s\in L^1(\Omega)$, $\na g^s\in L^2(\Omega)$ imply that $g^s\in L^{2d/(d-2)}(\Omega)$. Then, because of $s>(d-2)/(2d)$, it follows for sufficiently small $\eps>0$ (depending on $d$) that $g\in L^{2ds/(d-2)}(\Omega) \hookrightarrow L^{1+\eps s}(\Omega) = L^{ps}(\Omega)$. 

In contrast to the proof of Lemma \ref{lem.qs1}, we employ a Taylor expansion and the inequality $(a+b)^{p-2}\le 2^{p-2}(a^{p-2}+b^{p-2})$ for $a$, $b\ge 0$:
\begin{align*}
  f^p &= \bar{f}^p + p\bar{f}^{p-1}(f-\bar f)
  + p(p-1)\int_0^1\big(\bar f + \theta(f-\bar{f})\big)^{p-2}
  (1-\theta)\dd\theta(f-\bar f)^2 \\
  &\le \bar{f}^p + p\bar{f}^{p-1}(f-\bar f)
  + 2^{p-2}p(p-1)\int_0^1
  \big(\bar{f}^{p-2} + \theta^{p-2}(f-\bar{f})^{p-2}\big)
  (1-\theta)\dd\theta(f-\bar f)^2 \\
  &= \bar{f}^p + p\bar{f}^{p-1}(f-\bar f)
  + 2^{p-3}p(p-1)\bar{f}^{p-2}(f-\bar f)^2 + 2^{p-2}(f-\bar f)^p.
\end{align*} 
The second term on the right-hand side vanishes after integration over $\Omega$, leading to
\begin{align}\label{2.aux3}
  \|f\|_{L^p(\Omega)}^p \le |\Omega|\bar{f}^p
  + 2^{p-3}p(p-1)\bar{f}^{p-2}\|f-\bar f\|_{L^2(\Omega)}^2 
  + 2^{p-2}\|f-\bar f\|_{L^p(\Omega)}^p.
\end{align}
The second term on the right-hand side can be estimated using the Poincar\'e inequality, while for the third term, we first apply the Gagliardo--Nirenberg inequality with $\theta = d(p-2)/(2p)$, followed by the Poincar\'e--Wirtinger inequality:
\begin{align*}
  \|f-\bar f\|_{L^p(\Omega)}^p 
  &\le C_{GN}^p\|\na f\|_{L^2(\Omega)}^{p\theta}
  \|f-\bar f\|_{L^2(\Omega)}^{p(1-\theta)}
  + C_{GN}^p\|f-\bar f\|_{L^2(\Omega)}^p \\
  &\le C_{GN}^p(C_{P}^{p(1-\theta)}+C_{P}^p)\|\na f\|_{L^2(\Omega)}^p.
\end{align*}
The condition $s>(d-2)/(2d)$ guarantees that $\theta<1$. 

It remains to estimate $|\Omega|\bar{f}^p$ in \eqref{2.aux3}. As in the proof of Lemma \ref{lem.qs1}, Jensen's inequality shows that $|\Omega|\bar{f}^p = (\overline{g^s}/\bar{g}^s)^p\le 1$ (since $s<1$). Hence, we deduce from \eqref{2.aux3} and $\bar{f}^{p-2}\le|\Omega|^{2/p-1}$ that
\begin{align*}
  \|f\|_{L^p(\Omega)}^p 
  &\le 1 + 2^{p-3}p(p-1)C_P^2|\Omega|^{2/p-1}\|\na f\|_{L^2(\Omega)}^2
  + 2^{p-2}C_{GN}^p(C_{P}^{p(1-\theta)}+C_{P}^p)
  \|\na f\|_{L^2(\Omega)}^p \\
  &= 1 + a_1z + a_3z^{p/2},
\end{align*}
where $z = \|\na f\|_{L^2(\Omega)}^2$, $a_1 = 2^{p-3}p(p-1)C_P^2|\Omega|^{2/p-1}$, $a_3 = 2^{p-2}C_{GN}^p(C_{P}^{p(1-\theta)}+C_{P}^p)$. It follows from Lemma \ref{lem.ineq} that
\begin{align*}
  \int_\Omega g\log\frac{g}{\bar g}\dx 
  &\le \frac{|\Omega|\bar g}{\eps s}\log(1 + a_1z + a_3z^{p/2})
  \le \frac{|\Omega|\bar g}{\eps s}K_2 
  \|\na f\|_{L^2(\Omega)}^2 \\
  &= \frac{1}{\eps s}|\Omega|^{1-2/p}K_2\bar{g}^{-2s}
  \|\na g^s\|_{L^2(\Omega)}^2,
\end{align*}
where $K_2 = \max\{a_1+a_3,p/2-1\}$ only depends on $|\Omega|$, $d$, $s$, $C_P$, and $C_{GN}$. Setting $C_2(s)=|\Omega|^{1-2/p}K_2/(\eps s)$ completes the proof.
\end{proof}

\begin{lemma}
	Let $q>1$ and $s>(d-2)/(2d)$. Then there exists $C_3(q,s)>0$, only depending on $|\Omega|$, $d$, $q$, $s$, $C_P$, and $C_{GN}$ such that for all nonnegative functions $g\in L^q(\Omega)$ satisfying $\na g^{qs}\in L^2(\Omega)$,
\begin{align*}
  \int_\Omega(g^q-\bar{g}^q)\dx  \le C_3(q,s)
  \bigg(\fint_\Omega g^q \dx \bigg)^{1-2s}\int_\Omega|\na g^{qs}|^2 \dx ,
\end{align*}
where $C_3(q,s)=(q-1)C_i(s)$ with $i=1$ if $s>1/2$ and $i=2$ if $s\le 1/2$. 
\end{lemma}

This lemma generalizes Lemma \ref{lem.q} from $s=1$ to more general values of $s$. It proves Theorem \ref{thm.csi} for $q>1$, $s'>q(d-2)/(2d)$ after setting $s=s'/q$. 

\begin{proof}
We have shown in the proof of Lemma \ref{lem.q} (see \eqref{2.H}) that
\begin{align*}
  \frac{1}{q-1}\fint_\Omega(g^q-\bar{g}^q)\dx  
  = \frac{1}{q-1}\bigg(\fint_\Omega g^q\dx  
  - \bigg(\fint_\Omega g\dx \bigg)^q\bigg)
  \le \fint_\Omega g^q\log\frac{g^q}{\overline{g^q}}\dx .
\end{align*}
Then we deduce from Lemmas \ref{lem.qs1} or \ref{lem.qs2} with $g$ replaced by $g^q$ that 
\begin{align*}
  \int_\Omega(g^q-\bar{g}^q)\dx  
  \le (q-1)C_i(s)(\overline{g^q})^{1-2s}\int_\Omega|\na g^{qs}|^2 \dx ,
\end{align*}
where $i=1$ if $s>1/2$ and $i=2$ of $s\le 1/2$. 
\end{proof}

We can also treat the case  $s=q(d-2)/(2d)$, at least if $q>1$.

\begin{lemma}
	Let $q>1$ and $s=q(d-2)/(2d)$. Then, for all nonnegative functions $g\in L^q(\Omega)$ satisfying $\na g^{s}\in L^2(\Omega)$,
\begin{align*}
  \int_\Omega(g^q-\bar{g}^q)\dx  
  \le C_4(q,s) \|g\|_{L^q(\Omega)}^{q-2s}\int_\Omega|\na g^{s}|^2 \dx,
\end{align*}
where $C_4(q,s)>0$ depends on the constant of the Sobolev embedding $H^1(\Omega)\hookrightarrow L^{2d/(d-2)}(\Omega)$ as well as on $q$, $s$, and $C_P$. 
\end{lemma}

\begin{proof}
By Taylor expansion for $f(z)=z^{q/s}$,
\begin{align*}
  \int_\Omega&(g^q-\bar{g}^q)\dx 
  = \int_\Omega(f(g^s) - f(\overline{g^s}))\dx 
  - \int_\Omega(f(\bar{g}^s) - f(\overline{g^s}))\dx  \\
  &= \int_\Omega f'(\overline{g^s})(g^s-\overline{g^s})\dx 
  + \int_\Omega\int_0^1 
  f''\big(\overline{g^s}+\theta(g^s-\overline{g^s})\big)
  (g^s-\overline{g^s})^2(1-\theta)\dd\theta \dx  \nonumber \\
  &\phantom{xx}- \int_\Omega\int_0^1 f'\big(\overline{g^s}
  + \theta(\bar{g}^s- \overline{g^s})\big)(\bar{g}^s-\overline{g^s})
  \dd\theta \dx. \nonumber 
\end{align*}
The first integral on the right-hand side vanishes, leading to
\begin{align*}
  \int_\Omega(g^q-\bar{g}^q)\dx 
  &= \frac{q}{s}\bigg(\frac{q}{s}-1\bigg)\int_\Omega\int_0^1
  \big(\overline{g^s} + \theta(g^s-\overline{g^s})\big)^{q/s-2}
  (g^s-\overline{g^s})^2(1-\theta)\dd\theta \dx  \\
  &\phantom{xx}- \frac{q}{s}\int_\Omega\int_0^1
  \big(\overline{g^s} + \theta(\bar{g}^s -\overline{g^s})\big)^{q/s-1}
  (\bar{g}^s-\overline{g^s})\dd\theta \dx =: I_1 + I_2.
\end{align*}
The term $I_1$ is estimated as
\begin{align*}
  I_1 &\le 2^{q/s-2}\frac{q}{s}\bigg(\frac{q}{s}-1\bigg)
  \int_\Omega\int_0^1\big((\overline{g^s})^{q/s-2}
  (g^s-\overline{g^s})^2 
  + \theta^{q/s-2}(g^s-\overline{g^s})^{q/s}\big)(1-\theta)d\theta\dx \\
  &= 2^{q/s-3}\frac{q}{s}\bigg(\frac{q}{s}-1\bigg)
  (\overline{g^s})^{q/s-2}\int_\Omega(g^s-\overline{g^s})^2\dx 
  + 2^{q/s-2}\int_\Omega(g^s-\overline{g^s})^{q/s}\dx.
\end{align*}
Observing that $q/s = 2d/(d-2)>2$ and using the Sobolev embedding $H^1(\Omega)\hookrightarrow L^{2d/(d-2)}(\Omega) = L^{q/s}(\Omega)$ with constant $C_S>0$, we can write
\begin{align*}
  \int_\Omega(g^s-\overline{g^s})^{q/s}\dx
  &= \|g^s-\overline{g^s}\|_{L^{q/s}(\Omega)}^2
  \|g^s-\overline{g^s}\|_{L^{q/s}(\Omega)}^{q/s-2} \\
  &\le C C_S\|g\|_{L^q(\Omega)}^{q-2s}\|\na g^s\|_{L^2(\Omega)}^2,
\end{align*}
and the constant $C$ depends only on $q$, $s$, and $\Omega$. Together with the Poincar\'e--Wirtinger inequality, this gives
\begin{align*}
  I_1 \le C( q,s,C_P, C_S)\|g\|_{L^q(\Omega)}^{q-2s}
  \|\na g^s\|_{L^2(\Omega)}^2.
\end{align*}

We turn to the second term $I_2$. If $s\le 1$, the concavity of $z\mapsto z^s$ yields $I_2\le 0$. If $s>1$, we can use Lemma \ref{lem.q} to estimate 
\begin{align*} 
  I_2&\le C(s,q)\bigg(\int_\Omega g^s \dx\bigg)^{q/s- 1}C_p^2 \bigg(\fint_\Omega g^s\dx\bigg)^{-1}\int_\Omega|\na g^s|^2 \dx \\
  & \le C(s,q) \|g\|_{L^q(\Omega)}^{q-2s} 
  \|\na g^s\|_{L^2(\Omega)}^2.
\end{align*} 
This shows that 
\begin{align*}
  \int_\Omega(g^q-\bar{g}^q)\dx 
  &\le C(q, s, C_S, C_P)\|g\|_{L^q(\Omega)}^{q-2s}
  \|\na g^s\|_{L^2(\Omega)}^2. 
\end{align*}
finishing the proof.
\end{proof}

\begin{remark}[Optimality of the lower bound]\label{rem.opt}\rm
We claim that the lower bound $s\ge q(d-2)/(2d)$ is optimal. Indeed, assume that $s<q(d-2)/(2d)$. Let (without loss of generality) $B_1\subset\Omega$ be a unit ball around the origin and $n=1$, let $r<1$, and set $g_r(x)=(1-|x|/r)_+^{1/s}$ for $x\in B_1$. A computation shows that the left-hand side of the convex Sobolev inequality in Theorem \ref{thm.csi} behaves like $\mbox{LHS}\sim r^d$ as $r\to 0$, while the right-hand side behaves like $\mbox{RHS}\sim r^{d(q-2s)/q+d-2}=r^{2(d-ds/q-1)}$. Therefore, the quotient behaves like $\mbox{RHS}/\mbox{LHS}\sim r^{d-2-2ds/q}\to 0$ as $r\to 0$, since $s<q(d-2)/(2d)$. Thus, no positive constant can exist in this case.

Another justification of the optimality comes from the critical exponent of the fast-diffusion equation $\pa_t u=\diver(u^{m-1}\na u)$, giving $A(u)=u^{m-1}$. It is well known that the mass-conservative fast-diffusion regime is $m > m_c:=(d-2)/d$ (the case $m<(d-2)/d$ corresponds to the extinction regime; see \cite{BeFi24}). Then $h_1''(u)A(u)=u^{m-2}$. Hypothesis (H3) is satisfied for $s=m/2$, and the bound $s>(d-2)/(2d)$ is equivalent to $m>(d-2)/d$, which yields exactly the critical exponent $m_c$.
\end{remark}


\section{Variant of the Csisz\'ar--Kullback inequality}\label{sec.cki}

We prove the following result, which extends the Csisz\'ar--Kullback inequality to the case $q\ge 2$.

\begin{proposition}\label{prop.cki}
Let $h$ be a smooth convex function on $\dom$ such that there exists $c_h>0$ with $z^Th''(u)z\ge c_0\sum_{i=1}^n u_i^{q-2}z_i^2$ for all $u\in\dom$ and $z\in\R^n$. Then for all suitable functions $u$ and $v$,
\begin{align*}
  \int_\Omega h(u|v)\dx 
  \ge \begin{cases}\displaystyle
  \frac{1}{2\bar v}\|u-v\|_{L^1(\Omega)}^2 &\quad\mbox{if }q=1, 
  \\[8pt]
  \displaystyle
  \frac{c_0}{2}\max\{\|u\|_{L^q(\Omega)},\|v\|_{L^q(\Omega)}\}^{q-2}
  \|u-v\|_{L^q(\Omega)}^2 &\quad\mbox{if }1<q<2, \\
  \displaystyle
  \frac{c_0}{q}\sum_{i=1}^n\int_\Omega v_i^{q-2}(u_i-v_i)^2 \dx 
  &\quad\mbox{if }q\ge 2,
  \end{cases}
\end{align*}
where $h(u|v)$ is defined in \eqref{1.hrel}. 
\end{proposition}

\begin{proof}
The proof for $q=1$ is given, e.g., in \cite[Theorem A.2]{Jue16}, while the case $1<q<2$ was established in \cite{CCD02}. The case $q\ge 2$ follows from a Taylor expansion. Indeed, let $\phi(\theta) = h(v+\theta(u-v))$ for $0\le \theta\le 1$. Then
\begin{align*}
  h(u|v) &= \phi(1) - \phi(0) - \phi'(0)
  = \int_0^1\phi''(\theta)(1-\theta)\dd\theta \\
  &= \int_0^t\sum_{i,j=1}^n\frac{\pa^2 h}{\pa u_i\pa u_j}
  (v+\theta(u-v))(u_i-v_i)(u_j-v_j)(1-\theta)\dd\theta \\
  &\ge c_0\sum_{i=1}^n\int_0^1(v_i+\theta(u_i-v_i))^{q-2}
  (u_i-v_i)^2(1-\theta)\dd\theta \\
  &\ge c_0\sum_{i=1}^n\int_0^1(v_i-\theta v_i)^{q-2}(u_i-v_i)^2
  (1-\theta)\dd\theta \\
  &= c_0\sum_{i=1}^n v_i^{q-2}(u_i-v_i)^2
  \int_0^1(1-\theta)^{q-1}\dd\theta 
  = \frac{c_0}{q}\sum_{i=1}^n v_i^{q-2}(u_i-v_i)^2,
\end{align*}
completing the proof.
\end{proof}


\section{Proof of Theorem \ref{thm.time}}\label{sec.time}

Entropy inequality \eqref{1.ei} and Hypothesis (H3) show that
\begin{align*}
  \frac{\dd}{\dd t}\int_\Omega h_q(u|\bar u)\dx 
  &= \frac{\dd}{\dd t}\int_\Omega h_q(u)\dx  
  \le -\int_\Omega \na u^T:h_q''(u)A(u)\na u \dx  \\
  &\le -\frac{c_A}{s^2}\sum_{i=1}^n\int_\Omega|\na u_i^s|^2 \dx .
\end{align*}
We infer that $\int_\Omega h_q(u(t))\dx \le \int_\Omega h_q(u^0)\dx $ and consequently,
\begin{align}\label{4.uq}
  \sum_{i=1}^n\int_\Omega u_i^q(t)\dx 
  \le \sum_{i=1}^n\big(\|u_i^0\|_{L^1(\Omega)}
  + \|u_i^0\|_{L^q(\Omega)}^q\big) \le C(\|u^0\|_{L^q(\Omega)}),
\end{align}
as the $L^1(\Omega)$ norm of $u_i(t)$ equals the $L^1(\Omega)$ norm of $u_i^0$. It follows from Theorem \ref{thm.csi} that
\begin{align*}
  \sum_{i=1}^n\int_\Omega |\na u_i^s|^2 \dx 
  &\ge C(q,s,\|u\|_{L^q(\Omega)})\int_\Omega h_q(u|\bar u)\dx.
\end{align*}
We estimate the $L^q(\Omega)$ norm of $u$. If $q=1$, this norm equals the $L^1(\Omega)$ norm of $u_i^0$. Let $q>1$. Then interpolation yields 
$\bar{u}_i\le C\|u_i\|_{L^q(\Omega)}$, which provides a lower bound, since $\bar{u}_i$ equals the $L^1(\Omega)$ norm of $u_i^0$. The upper bound follows from \eqref{4.uq}. Thus, the constant $C(q,s,\|u\|_{L^q(\Omega)})$ depends only on the $L^q(\Omega)$ norm of $u^0$. This shows, for some constant $C>0$, that 
\begin{align*}
  \frac{\dd}{\dd t}\int_\Omega h_q(u|\bar u)\dx 
  + C(\|u^0\|_{L^q(\Omega)})
  \int_\Omega h_q(u|\bar u)\dx  \le 0.
\end{align*}
Then Gronwall's inequality gives
\begin{align*}
  \int_\Omega h_q(u(t)|\bar u)\dx 
  \le e^{-\lambda t}\int_\Omega h_q(u^0|\bar u)\dx ,
\end{align*}
where $\lambda=C(\|u^0\|_{L^q(\Omega)})$. The convergence in the $L^r(\Omega)$ norm follows from the Csisz\'ar--Kullback inequality in Proposition \ref{prop.cki}. 


\section{Examples}\label{sec.exam}

We present two examples for which exponential equilibration was not known yet.

\subsection{Superlinear SKT model}

The dynamics of segregating population species is governed by equations \eqref{1.eq} with the diffusion coefficients
\begin{align*}
  A_{ij}(u) = \delta_{ij}p_i(u) + u_i\frac{\pa p_i}{\pa u_j}(u),
  \quad p_i(u) = \sum_{k=1}^n a_{ik}u_k^s,
\end{align*}
where $i,j=1,\ldots,n$, $s>\max\{0,1-2/d\}$, $d\ge 3$ (to simplify), $(a_{ij})$ is symmetric, positive definite, and
\begin{align*}
  \mu = \min_{i=1,\ldots,n}\bigg(a_{ii}-\frac{s-1}{s+1}
  \sum_{j=1,\,j\neq i}^n a_{ij}\bigg) > 0.
\end{align*}
The existence of a global nonnegative weak solution $u$ to \eqref{1.eq}--\eqref{1.bic} satisfying \eqref{1.regul} was shown in \cite[Theorem 2]{CDJ18}. Using the techniques of \cite{HeJu26}, we can even verify the entropy inequality
\begin{align*}
  \frac{\dd}{\dd t}\int_\Omega h_s(u)\dx + C(\mu,s)
  \sum_{i=1}^n\int_\Omega |\na u_i^s|^2 \dx  \le 0,
\end{align*}
where $C(\mu,s)>0$ can be computed explicitly. We infer from Theorem \ref{thm.time} for $q=s$ the existence of $\lambda>0$ and $C>0$ such that
\begin{align*}
  \|u(t)-\bar u\|_{L^r(\Omega)} \le Ce^{-\lambda t}, \quad t>0,\  
  r = \min\{2,s\}.
\end{align*}

\subsection{Nonlocal Busenberg--Travis model}

The generalized Busenberg--Travis equations with Brinkman-type law, 
\begin{align*}
  \pa_t u_i - \Delta u_i + \diver(u_iv_i) = 0, \quad 
  -\Delta v_i + v_i = -\na p_i(u), \quad p_i(u) = \sum_{k=1}^n a_{ik}u_k^q,
\end{align*}
do not exactly fit into the framework \eqref{1.eq}, but our method can still be applied. The existence of a global weak solution $u$ for some range of values $q>0$ was proved in \cite{HiJu26}. The entropy inequality is given by
\begin{align*}
  \frac{\dd}{\dd t}\int_\Omega h_q(u)\dx  + \frac{4}{q}
  \sum_{i=1}^n\int_\Omega |\na u_i^{q/2}|^2 \dx  \le 0.
\end{align*}
Thus, Theorem \ref{thm.time} with $q\ge 1$, $s=q/2$ shows the existence of a number $\lambda>0$ such that
\begin{align*}
  \|u(t)-\bar u\|_{L^r(\Omega)} \le Ce^{-\lambda t}, \quad t>0,\
  r=\min\{2,q\}.
\end{align*} 


\end{document}